\documentclass[11pt]{article}
\usepackage{latexsym,color,amsmath,amsthm,amssymb,amscd,amsfonts,stmaryrd,mathabx,wasysym}

\newtheorem{lemma}{Lemma}[section]
\newtheorem{proposition}[lemma]{Proposition}
\newtheorem{remark}[lemma]{Remark}
\newtheorem{theorem}[lemma]{Theorem}
\newtheorem{definition}[lemma]{Definition}

\newtheorem{conj}[lemma]{Conjecture}
\newtheorem{ques}[lemma]{Question}

\newtheorem*{remark*}{Remark}

\begin{document}
\title {Asymptotic Equivalence of Symplectic Capacities }
\author{Efim D. Gluskin and Yaron Ostrover }
\date{}
\maketitle
\begin{abstract}  

A long-standing conjecture states that all normalized symplectic capacities coincide on the class of convex subsets of ${\mathbb R}^{2n}$. In this note we focus on an asymptotic (in the dimension) version of this conjecture, and show that when restricted to the class of centrally symmetric convex bodies in ${\mathbb R}^{2n}$,
 several symplectic capacities, including the Ekeland--Hofer--Zehnder capacity, the displacement energy capacity, and the cylindrical capacity, are all equivalent up to an absolute constant.

\end{abstract}

\section{Introduction} \label{sec:Int}
Consider the space ${\mathbb R}^{2n}$ equipped both with the standard symplectic
form $\omega= dp \wedge dq$, and with the standard inner product $\langle \cdot, \cdot \rangle$.
Note that under the usual identification between ${\mathbb R}^{2n}$ and ${\mathbb C}^n$, 
these two structures are the real and the imaginary parts, respectively, of the standard Hermitian inner product in ${\mathbb C}^n$. Moreover,  one has that
 $\omega(v, u) = \langle v,Ju \rangle$, where $J$ is the standard complex structure in ${\mathbb R}^{2n} \simeq {\mathbb C}^n$. 
 Symplectic
capacities, whose axiomatic definition below is due to Ekeland
and Hofer~\cite{EH1}, are numerical invariants which roughly speaking measure the symplectic size of sets. More precisely, let  $B^{2n}(r)$ stand for the Euclidean open ball of radius $r$, and $Z^{2n}(r)$ for the cylinder $B^{2}(r) \times {\mathbb C}^{n-1}$. 
\begin{definition} \label{Def-sym-cap}
A  symplectic capacity on $({\mathbb R}^{2n},\omega)$ associates
to each   subset $U \subset {\mathbb R}^{2n}$ a number $c(U) \in
[0,\infty]$ such that the following hold:

\noindent $(P1)$ \, $c(U) \leq c(V)$ whenever $U \subseteq V$
 (monotonicity),

\noindent $(P2)$ \, $c \big (\psi(U) \big )= |\alpha| \, c(U)$ for
$\psi \in {\rm Diff} ( {\mathbb R}^{2n} )$ such that
$\psi^*\omega = \alpha \, \omega$ (conformality),

\noindent $(P3)$ \,   $0 < c \big (B^{2n}(r) \big ), {\rm and} \  c \big (Z^{2n}(r) \big ) < \infty$ (nontriviality).

\noindent Moreover, a symplectic capacity is said to be normalized if in addition it satisfies

\noindent $(P4)$ \, $c \big (B^{2n}(r) \big ) = c \big (Z^{2n}(r) \big ) = \pi r^2$ (normalization).
\end{definition}
Note that propery $(P2)$ implies that $c$ is a symplectic invariant which scales like a two-dimensional invariant, and 
$(P3)$ that symplectic capacities significantly differ from any volume related invariants.
The first examples of symplectic capacities were  constructed by Gromov in~\cite{G}, 
where he developed and used pseudoholomorphic curve techniques to prove a striking symplectic rigidity result, nowadays known as Gromov's ``non-squeezing theorem". It states that one cannot map a ball inside a thinner cylinder by a symplectic embedding.  More precisely, the theorem asserts that if $r < 1$, there is no symplectic embedding of
the unit ball $B^{2n}$ into the cylinder $Z^{2n}(r)$. 
This naturally leads to the definition of two normalized symplectic capacities: 
the Gromov width, given by $\underline c(U)=\sup\{\pi r^2 \, | \, B^{2n}(r) \stackrel{\rm s} \hookrightarrow U \} $; and the 
cylindrical capacity,  $\overline c(U) = \inf\{\pi r^2  \, | \, U \stackrel{\rm s} \hookrightarrow Z^{2n}(r) \} $. Here $\stackrel{\rm s} \hookrightarrow$ stands for symplectic embedding. 
It is not hard to verify that these two capacities are the smallest and largest  possible normalized symplectic capacities, respectively.

Shortly after Gromov's work~\cite{G} many other symplectic capacities were
constructed, reflecting different geometrical and dynamical properties. Among these are the Hofer--Zehnder capacity~\cite{HZ,HZ1}, 
the Ekeland--Hofer capacities~\cite{EH1,EH2}, the displacement
energy~\cite{H1}, the Floer--Hofer capacity~\cite{FH,FHW}, spectral capacities~\cite{FGS,Oh,V2}, and more recently,
Hutchings'  embedded contact homology (ECH) capacities~\cite{Hu1}. 
These quantities play an important role in symplectic geometry, and  their properties, interrelations, and applications to symplectic topology and Hamiltonian dynamics are  intensively studied (see e.g.,~\cite{CHLS} and~\cite{Mc1} for two excellent surveys).

In the two-dimensional case, Siburg~\cite{Sib} showed that any symplectic capacity of  a compact connected domain with smooth boundary $\Omega \subset {\mathbb R}^2$ equals its Lebesgue measure. 
In higher dimensions symplectic capacities do not coincide in general. 
A  theorem  by Hermann~\cite{Her} states that for any $n \geq 2$ there
is a bounded star-shaped 
domain $S \subset {\mathbb R}^{2n}$  with
cylindrical capacity $\overline c(S) \geq 1$, and arbitrarily small Gromov
width $\underline c(S)$.  Still, for a large class of sets in
${\mathbb R}^{2n}$, including ellipsoids, polydiscs, and convex Reinhardt
domains, all normalized symplectic capacities coincide~\cite{Her}. 
In~\cite{V} Viterbo showed that for any bounded convex set $K$
 of ${\mathbb R}^{2n}$ one has  $\overline c(K) \leq 4n^2 \underline c(K)$. 
Moreover, it was conjectured~\cite{Her,H2,V} that:
\begin{conj}   \label{conj-all-cap-coincide}
For any convex body $K$ in
${\mathbb R}^{2n}$ one has $\underline c(K) = \overline c(K)$.
\end{conj}
Here, by a convex body we mean a compact convex subset of ${\mathbb R}^{2n}$ with non-empty interior. The above conjecture is particularly challenging due to the scarcity of examples of convex domains for which capacities have been computed. 
Moreover, an affirmative answer to Conjecture~\ref{conj-all-cap-coincide} would in particular implies 
Viterbo's volume-capacity conjecture~\cite{V}, and it was recently shown that the latter would in turn 
settle a 70-years old question in convex geometry known as the Mahler conjecture. For more information regarding these applications of Conjecture~\ref{conj-all-cap-coincide} see~\cite{AAKO} and~\cite{O}.  

A somewhat more modest question in the same direction (c.f. Problem 1.4 in~\cite{Her}, Problem 8 in~\cite{CHLS}, and Section 5 in~\cite{O}) is whether Conjecture~\ref{conj-all-cap-coincide} above
holds asymptotically in the dimension, i.e., 
\begin{ques} \label{ques-about-equiv} Is
there is an absolute constant $A > 0$ such that for
every convex body $K$ in ${\mathbb R}^{2n}$ one has $$ \overline c(K) \leq  A \underline c(K).$$
\end{ques} 

Here we will give a partial answer to this question.
Before we state our main result we need to recall the definition of the Ekeland--Hofer--Zehnder capacity.
The restriction of the symplectic form $\omega$ to a smooth closed
hypersurface $\mathcal S \subset {\mathbb R}^{2n}$ canonically defines a
1-dimensional subbundle, ${\rm ker}(\omega | {\mathcal S})$, whose integral
curves comprise the characteristic foliation of $\mathcal S$. 
In other words, a closed characteristic of $\mathcal S$  is an embedded circle in $\mathcal S$ tangent to the canonical line bundle
$$ {\mathfrak S}_{\mathcal S}= \{(x, \xi) \in T {\mathcal S} \, | \, \omega(\xi,\eta) = 0 \ {\rm for \ all} \ \eta \in T_x{\mathcal S }\}. $$
Recall that the symplectic action of a closed curve $\gamma$ is defined by $A(\gamma) = \int_{\gamma} \lambda, $ where $\lambda = pdq$ is the Liouville 1-form. The action spectrum of ${\mathcal S}$ is
$$ {\mathcal L}({\mathcal S}) = \{ |A(\gamma)| \, ; \,  \gamma \ {\rm is \ a \ closed \ characteristic \ on \ } {\mathcal S} \}.$$

In~\cite{EH1} and~\cite{HZ1} it was
proved that for a smooth convex 
body $K \subset {\mathbb R}^{2n}$, the two aforementioned 
Hofer--Zehnder and  Ekeland--Hofer  
capacities coincide, and are given by the minimal action over
all closed characteristics on the boundary of the 
body $K$, i.e.,
\begin{equation} \label{EHZ-cap-def} c_{_{\rm EH}}(K) = c_{_{\rm HZ}}(K) = \min \, {\cal L}(\partial K). \end{equation}
We remark that although the above definition of closed characteristics, as well
as the equalities in~$(\ref{EHZ-cap-def})$, were given only for the class of convex bodies with smooth boundary,
they can naturally be generalized to the class of convex sets in ${\mathbb R}^{2n}$ with nonempty interior (see e.g.,~\cite{AAO}).
In what follows, we refer to the coinciding Ekeland--Hofer
and Hofer--Zehnder capacities on this class as the Ekeland--Hofer--Zehnder capacity, and denote it by $c_{_{\rm EHZ}}$.

Our first result in the note is the following. Recall that a convex body $K \subset {\mathbb R}^n$ is said to be centrally symmetric if $K=-K$.
\begin{theorem} \label{main-thm-weak}
For every centrally symmetric convex body $K$ in $ {\mathbb R}^{2n}$,
\begin{equation*} \label{cap-eq-up-to-constnat} \overline c(K) \leq 4 c_{_{\rm EHZ}}(K). \end{equation*}
\end{theorem}

\begin{remark} {\rm Other symplectic capacities, like the spectral capacities $c_{\sigma}$, which are based on a choice of an action selector $\sigma$, and the displacement energy $d$, are known to be bigger than or equal to the Hofer--Zehnder capacity (see e.g., Section 2.3.4 in~\cite{CHLS}). Thus, it follows from Theorem~\ref{main-thm-weak} that on the class of symmetric convex sets in ${\mathbb R}^{2n}$, 
the normalized symplectic capacities $c_{_{\rm EHZ}}, d, c_{\sigma}$ and $\overline c$, 
are all coincide up to an absolute constant. 
}
\end{remark} 

In fact, we prove a slightly stronger result than Theorem~\ref{main-thm-weak}  which shows that for a centrally symmetric convex body $K \subset {\mathbb R}^{2n}$, the aforementioned symplectic capacities are all equivalent to yet another quantity associated with 
the body $K$. More precisely, for a convex body $K \subset {\mathbb R}^{2n}$ with $0 \in {\rm Int}(K)$, we denote by  $K^{\circ} = \{ y \in {\mathbb R}^{2n} \ | \ \langle x,y \rangle \leq 1, \ {\rm for \ every \ } x \in K \}$ the polar body of $K$\footnote[2]{As a matter of fact, the polar body $K^{\circ}$ should be defined as a subset of the dual space
of ${\mathbb R}^{2n}$. However, since we have fixed a scalar product in our setting, we will identify the latter space with ${\mathbb R}^{2n}$ itself.}. Moreover, we denote 
$$\|J\|_{_{{K^{\circ} \rightarrow K}}} : = \sup_{v,u \in K^{\circ}} \langle Jv,u \rangle.$$ 
To explain the reason for this notation, we remark 
that when the convex body $K$ is centrally symmetric, $\|J\|_{_{{K^{\circ} \rightarrow K}}}$ is the operator 
norm of the complex structure $J$, when the latter is considered as a linear map between the normed spaces $J : ({\mathbb R}^{2n},\| \cdot \|_{K^{\circ}}) \rightarrow ({\mathbb R}^{2n}, \| \cdot \|_{K})$, i.e.,
$$ \|J\|_{_{{K^{\circ} \rightarrow K}}} = \sup_{v,u \in K^{\circ}} \langle Jv,u \rangle =  
\sup_{v \, : \, \|v \|_ {K^{\circ}} \leq 1 } {\| Jv \|_{K}}.$$
Here we use the standard identification between normed spaces and centrally symmetric convex bodies, i.e., for a non-empty centrally symmetric convex body $K$ in ${\mathbb R}^{2n}$ we denote by $\| \cdot \|_K$ the norm on ${\mathbb R}^{2n}$ induced by $K$, that is, $\| \cdot \|_K = \inf \{r \, : \,  x \in rK\}$. 

\begin{theorem} \label{main-thm} For every centrally symmetric convex body $K$ in ${\mathbb R}^{2n}$,
\begin{equation} \label{equiv-of-cap}  {\frac {1} {\|J\|_{_{{K^{\circ} \rightarrow K}}}}} \leq c_{_{\rm EHZ}}(K) \leq \overline c(K) \leq  {\frac {4} {\|J\|_{_{{K^{\circ} \rightarrow K}}}}}.\end{equation} \end{theorem}
\begin{remark} \label{rem-about-lower-bd-non-sym} {\rm In fact, in the proof of Theorem~\ref{main-thm} we use the centrally symmetric assumption on the body $K$ only for the right-most inequality of~$(\ref{equiv-of-cap})$. The first two inequalities on the left-hand side hold for every convex body $K$ in ${\mathbb R}^{2n}$.} \end{remark}
Note that Theorem~\ref{main-thm-weak} follows immediately from Theorem~\ref{main-thm}. Moreover, we wish to emphasize  that Theorem~\ref{main-thm} provides in many cases an efficient way to approximate the numerical value of the capacities $c_{_{\rm EHZ}}(K)$ and $\overline c(K)$ (for centrally symmetric convex bodies), as the quantity  $ \|J\|_{_{{K^{\circ} \rightarrow K}}} $ is a-priori much easier to compute than the above mentioned symplectic capacities. 

Another by-product of Theorem~\ref{main-thm}, which may be of independent interest, concerns the equivalence of the cylindrical capacity and the Gromov width capacity 
with their linearized versions ${\overline c}_{\rm lin}$ and  $ \underline c_{\rm lin}$  respectively. The definitions of 
these two quantities are given in Definitions~\ref{def-linearized-upper-cap} and~\ref{def-linearized-Gromov-width} below.  It turns out that  for centrally symmetric convex bodies in ${\mathbb R}^{2n}$, the cylindrical capacity $\overline c$ is  asymptotically equivalent to its linearized version ${\overline c}_{\rm lin}$, while surprisingly enough, this is false for the Gromov width capacity. More precisely,

\begin{theorem} \label{thm-about-lin-capacities}
      For every centrally symmetric convex body $K$ in ${\mathbb R}^{2n}$,
$$ \overline c(K) \leq  \overline c_{\rm lin}(K) \leq 4  \overline c(K).$$ 
On the other hand, there exist  a centrally symmetric convex body $\widetilde K$ in ${\mathbb R}^{2n}$ such that
$$ \underline c_{\rm lin}(\widetilde K) \leq \pi, \ {while } \ \ \underline c(\widetilde K) \geq \sqrt {{\frac n 2}}. $$
 \end{theorem}
Note that an immediate corollary from Theorem~\ref{thm-about-lin-capacities} is that the linearized versions of the Gromov width and the cylindrical capacity are not asymptotically equivalent.

\noindent{{\bf Notations:}} We denote by ${\mathcal K}^n$ the class of convex bodies of ${\mathbb R}^n$, i.e.,  compact convex sets with non-empty interior. 
For $K \in {\mathcal K}^n$, we denote by $h_K : {\mathbb R}^n \rightarrow {\mathbb R}$ its support function given by $h_K(u) = \sup \{ \langle x , u \rangle :  x \in K \}$.
Also, we denote by $g_K : {\mathbb R}^n \rightarrow {\mathbb R}$ the gauge function $g_K(x) = \inf \{r  |  x \in rK \}$ associated with $K$. Note that when
$K$ is centrally symmetric, i.e., $K = -K$, the gauge function $g_K(x)$ is a norm, and is  
denoted by $\|x\|_K$.  Furthermore, when $0 \in {\rm int}(K),$ one has that $h_K = g_{K^{\circ}}$, where
$K^{\circ} = \{ y \in {\mathbb R}^{n} \ | \ \langle x,y \rangle \leq 1, \ {\rm for \ every \ } x \in K \}$ is the polar body of $K$. The Euclidean norm will be denoted by $| \cdot |$. Finally, we denote by ${\mathbb S}^n$ the unit sphere in ${\mathbb R}^{n+1}$, i.e., ${\mathbb S}^n = \{ x \in {\mathbb R}^{n+1} \ | \ |x|=1 \}$.

\noindent{{\bf Acknowledgments:}} The authors are grateful to Shiri Artstein-Avidan and Boaz Klartag 
for many stimulating discussions on various topics related to convex and symplectic geometry.
The second-named author was partially supported by the European Research Council (ERC)
under the European UnionÕs Horizon 2020 research and innovation programme, starting grant No. 637386, and by the ISF grant No.1274/14.

\section{Proof of Theorem~\ref{main-thm}} \label{Sec-Proof}

Note first that there is no loss of generality in assuming that in addition to being compact and with non-empty interior, all convex bodies considered
also have a smooth boundary, and contain the origin in their interior. 
Indeed, affine translations in ${\mathbb R}^{2n}$ are symplectomorphisms,
which accounts for the assumption that  the origin is in the interior. Secondly,  
once Theorem~\ref{main-thm} is proved
for smooth convex domains, the general case follows by standard approximation
arguments, as symplectic capacities are continuous on the class of convex bodies
with respect to the Hausdorff distance (see e.g.~\cite{MS} page 376).

Moreover, in what follows we will make repeated use of the following well-known geometric observation form convex geometry.
\begin{lemma} \label{convex-geom-lemma}
Let $ g_K$ be the gauge function associated with a smooth convex body $K$. 
Then, when restricted to the boundary $\partial K$, the gradient $\nabla g_K$ is a surjective map $\nabla g_K : \partial K \rightarrow \partial K^{\circ}$.
\end{lemma}

A proof of Lemma~\ref{convex-geom-lemma} can be found e.g., in Subsection 1.7.1 of~\cite{Schn}.
We turn now to the proof  of
Theorem~\ref{main-thm}, and start with the following proposition.

\begin{proposition} \label{lower-bound-for-ECH-cap}
For every smooth convex body $K \in {\mathcal K}^{2n}$,
\begin{equation*} \label{lower-bound-for-ech}  {\frac {1} {\|J\|_{_{{K^{\circ} \rightarrow K}}}}} \leq c_{_{\rm EHZ}}(K). \end{equation*}
\end{proposition}

To prove Proposition~\ref{lower-bound-for-ECH-cap} we first need some preparation.
Recall (see e.g., Chapter 1 of~\cite{HZ}) that the classical geometric problem of finding closed characteristics on $\partial K$ has the following
dynamical interpretation. If the boundary $\partial K$ is represented  as a regular  energy surface
$\{ x \in {\mathbb R}^{2n} \, | \, H(x) = 1 \}$  of a smooth Hamiltonian function $H : {\mathbb R}^{2n} \rightarrow {\mathbb R}$, 
then the restriction to $\partial K$ of the Hamiltonian vector field $X_H$, defined by $i_{X_H} \omega = -dH,$ is a section of the line bundle ${\mathfrak S}_{\partial K}$.
Thus, the images of the periodic solutions of the classical Hamiltonian
equation $\dot x = X_H(x) = J \nabla H(x)$ on $\partial K$ are precisely the closed characteristics of $\partial K$. 
In particular, the closed characteristics do not depend (up to parametrization)
on the choice of the Hamiltonian function. Indeed, if the energy surface can be represented as a regular
level set of some other function $F : {\mathbb R}^{2n} \rightarrow {\mathbb R}$, then $X_H = \alpha X_F$ on $\partial K$ for some scalar function $\alpha \neq 0$, and the corresponding Hamiltonian equations have the same solutions up to parametrization. Finally, note that for a smooth convex body $K$ the gauge function $g_K$ is a {\it defining function} for $K$, i.e., $K = g_K^{-1}([0,1]), \partial K = g_K^{-1}(1)$, and $1$ is a regular value of $g_K$.

\begin{lemma} \label{observation1} Let $\gamma : [0,T] \rightarrow \partial K$ be a solution of the Hamiltonian equation $\dot \gamma = J \nabla g_K(\gamma)$, with $\gamma(0)=\gamma(T)$. Then there exist $t_0 \in [0,T]$ such that $g_K(\gamma(t_0)-\gamma(0)) \geq 1$. 
\end{lemma}
\begin{proof}[{\bf Proof of Lemma~\ref{observation1}}]
It follows immediately from the assumptions that 
\begin{equation*} 0 =  \int_0^T \dot \gamma (t)dt = \int_0^T J \nabla g_K(\gamma(t))dt =J  \int_0^T \nabla g_K(\gamma(t))dt.
\end{equation*}
From this one can conclude that
\begin{equation*}  \int_0^T \langle \nabla g_K(\gamma(t)) , \gamma(0) \rangle dt = 0.
\end{equation*}
In particular, this implies that there exists $t_0 \in [0,T]$ such that 
\begin{equation}  \label{eq1-proof-lemma}  \langle \nabla g_K(\gamma(t_0)) , \gamma(0) \rangle \leq 0.
\end{equation}
Next, 
from Lemma~\ref{convex-geom-lemma}  it follows that 
$\nabla g_K (\gamma(t_0)) \in \partial K^{\circ}$,
and we obtain that 
\begin{equation} \label{eq2-proof-lemma} g_K(\gamma(t_0)-\gamma(0))=\sup \{ \langle \gamma(t_0)-\gamma(0), u \rangle \, | \, u \in K^{\circ} \} \geq \langle \gamma(t_0)-\gamma(0), \nabla g_K (\gamma(t_0)) \rangle.
\end{equation}
Finally, from Euler's homogeneous function theorem it follows that  
 for every $x \in \partial K$, one has $\langle x , \nabla g_K(x) \rangle = g_K(x)=1$, and hence the combination of this fact together with inequalities~$(\ref{eq1-proof-lemma})$ and~$(\ref{eq2-proof-lemma})$ completes the proof of the lemma.
\end{proof}

\begin{proof}[{\bf Proof of Proposition~\ref{lower-bound-for-ECH-cap}}]


Let $\gamma : [0,T] \rightarrow \partial K$ be a closed characteristic on the boundary $\partial K$, i.e., a solution of the Hamiltonian equation $\dot \gamma = J \nabla g_K(\gamma)$, with $\gamma(0)=\gamma(T)$. Note that 
\begin{equation}\label{action-period-relation}
A(\gamma) = {\frac 1 2} \int_0^T \langle J \gamma(t), \dot \gamma(t) \rangle \, dt = {\frac 1 2}  \int_0^T \langle \gamma(t), \nabla g_K(\gamma(t)) \rangle \, dt = {\frac T 2}.
\end{equation}
It follows from Lemma~\ref{observation1}, the subadditivity property of $g_K$, and the definition of $\gamma$, that,
\begin{equation} \label{observation2}
1 \leq g_K \left (\int_0^{t_0} \dot \gamma(t) dt \right) \leq \int_0^{t_0} g_K ( \dot \gamma(t) ) dt = \int_0^{t_0} g_K ( J \nabla g_K(\gamma(t))) dt.
\end{equation}
On the other hand, it follows from the definition of an operator norm that 
\begin{equation} \label{observation3}
\int_0^{t_0} g_K ( J \nabla g_K(\gamma(t))) dt \leq \int_0^{t_0} \|J\|_{_{{K^{\circ} \rightarrow K}}} \, g_{K^{\circ}} (\nabla g_K(\gamma(t))) \,dt. 
\end{equation}
The combination of~$(\ref{observation2})$,~$(\ref{observation3})$, and Lemma~\ref{convex-geom-lemma} gives
\begin{equation} \label{observation4}
1 \leq \int_0^{t_0} \|J\|_{_{{K^{\circ} \rightarrow K}}} \, g_{K^{\circ}} (\nabla g_K(\gamma(t))) \,dt = \int_0^{t_0} \|J\|_{_{{K^{\circ} \rightarrow K}}} \, dt, 
\end{equation}
and thus we obtain that
\begin{equation} \label{t_0-est1}
{\frac 1 {\|J\|_{_{{K^{\circ} \rightarrow K}}}} } \leq t_0.
\end{equation}
Note that since $\gamma(0)=\gamma(T)$, repeating the same arguments as above (this time, integrating in~$(\ref{observation2})$,~$(\ref{observation3})$, and~$(\ref{observation4})$ between $t_0$ and $T$) we obtain also that 
\begin{equation} \label{t_0-est2}
{\frac 1 {\|J\|_{_{{K^{\circ} \rightarrow K}}}} } \leq T-t_0.
\end{equation}
From~$(\ref{action-period-relation})$ it follows that $\min \{t_0, T-t_0\} \leq T/2 = A(\gamma),$ and since, by definition, the capacity $c_{_{\rm EHZ}}(K)$ is defined to be the minimal action of closed characteristics on the boundary $\partial K$, we conclude from~$(\ref{t_0-est1})$ and~$(\ref{t_0-est2})$  that, $$ {\frac {1} {\|J\|_{_{{K^{\circ} \rightarrow K}}}}}  \leq c_{_{\rm EHZ}}(K).$$
This completes the proof of the proposition.  
\end{proof}

To describe the second ingredient in the proof of Theorem~\ref{main-thm} we need to introduce one more definition.
It is known (see e.g., Appendix C in~\cite{Schle}) that for a Lebesgue measurable set ${\mathcal U}  \subset {\mathbb R}^{2n}$,
$$ \overline c({\mathcal U} ) = \inf_{\varphi } {\rm Area} \bigl (\pi(\varphi({\mathcal U} )) \bigr),$$
where $\pi$ is the orthogonal projection to the complex line 
$E = \{ z \in {\mathbb C}^n \, | \  z_j = 0 \  {\rm for} \  j \neq 1\}$, and the infimum is taken over all symplectic embeddings $\varphi$ of ${\mathcal U} $ into ${\mathbb R}^{2n}$. Recall that with our notations, under the natural identification ${\mathbb R}^{2n} \simeq {\mathbb C}^n$ one has that $z_j = q_j + ip_j$. 
Thus, a nature way to ``linearize" the cylindrical capacity $\overline c$ is as follows. Let $ {\rm ISp}(2n)$ be the affine symplectic group, defined as the semi-direct product ${\rm Sp}(2n) \ltimes {\rm T}(2n)$ of the linear symplectic group and the group of translations in ${\mathbb R}^{2n}$.

\begin{definition} \label{def-linearized-upper-cap} The linearized cylindrical capacity ${\overline c}_{\rm lin}$ of a set ${\mathcal U} \subset {\mathbb R}^{2n}$ is defined as
$$ {\overline c}_{\rm lin}({\mathcal U} ) = \inf_{S} {\rm Area} \bigl (\pi(S({\mathcal U} )) \bigr),$$
where the infimum is taken over all affine symplectic maps $S \in {\rm ISp}(2n)$. 
\end{definition}

Now, the second main ingredient in the proof of Theorem~\ref{main-thm} is the following:
\begin{proposition} \label{upper-bound-for-cylindrical-cap}
For every  centrally symmetric convex body $K \in {\mathcal K}^{2n}$,
\begin{equation} \label{bound-for-c-upper}  
\overline c(K) \leq  \overline c_{\rm lin}(K) \leq {\frac {4} {\|J\|_{_{{K^{\circ} \rightarrow K}}}}}.
\end{equation}
\end{proposition}
To establish Proposition~\ref{upper-bound-for-cylindrical-cap} we shall need the following geometric observation.
For $v \in {\mathbb R}^{2n}$, we denote by $K_v$ the section $K \cap \{ v\}^{\perp}$, and by $\| \cdot \|_{K_v^{\circ}}$ the semi-norm defined by $$\| w \|_{K_v^{\circ}} = \sup \{ \langle w,y \rangle \, | \, y \in K_v \}.$$ 
\begin{lemma} \label{gelm-obser} For a symmetric convex body $K \in {\mathcal K}^{2n}$, a linear  symplectic map $S \in {\rm Sp}(2n)$, and 
 the orthogonal projection $\pi$ to the complex line $E = \{ z \in {\mathbb C}^n \, | \  z_j = 0 \, {\rm for} \ j \neq 1\}$ defined above, one has
\begin{equation} \label{area-estimate}  {\rm Area} \bigl (\pi(S({K} )) \bigr) \leq  4 \| S^Te \|_{K^{\circ}}  \| S^T Je \|_{K_{v}^{\circ}},
\end{equation} 
where $S^T$ stands for the transpose of the matrix $S$, $e$ is a unit vector parallel to the $q_1$-axis, and $v = S^Te$.
\end{lemma}
\begin{proof}[{\bf Proof of Lemma~\ref{gelm-obser}}]
The lemma follows from a much more general result by Rogers and Shephard~\cite{RS}, which states that for every symmetric convex body $K \subset {\mathbb R}^n$,
one has 
\begin{equation} \label{RS-ineq} {\rm Vol_n}(K) \leq \Bigl ({\rm Vol_k} (\pi_E(K)) {\rm Vol_{n-k}}(K \cap E^{\perp}) \Bigr)^{1/k} \leq {\binom{n}{k}}^{1/k} {\rm Vol_n}(K) , \end{equation}
for every $k$-dimensional subspace $E$ of ${\mathbb R}^{n}$, where $\pi_E$ stands for the orthogonal projection on the subspace $E$.  We remark that we use only the case where $n=2$ and $k=1$, in 
which inequality~$(\ref{RS-ineq})$ is an elementary geometric fact. It can be easily checked that the right-hand side of~$(\ref{area-estimate})$ exactly equals the product of the length of the projection 
of $\pi(SK)$ to the $q_1$-axis, and the length of the intersection of $\pi(SK)$ with the $p_1$-axis. 
\end{proof}
We are now in a position to prove Proposition~\ref{upper-bound-for-cylindrical-cap}.
 \begin{proof}[{\bf Proof of Proposition~\ref{upper-bound-for-cylindrical-cap}}] 
Note that, by definition, for every measurable set ${\mathcal U} \subset {\mathbb R}^{2n}$ one has $  \overline c({\mathcal U} ) \leq {\overline c}_{\rm lin}({\mathcal U} )$, and hence the left-hand side inequality in~$(\ref{bound-for-c-upper})$ holds. 
Next, we recall the easily verified fact that for any $v,w \in {\mathbb R}^{2n}$  such that $\omega(v,w)=1$, there exists a linear symplectic map $S \in {\rm Sp}(2n)$ such that $v=S^Te$ and $w=S^TJe$, where as before, $e$ is a unit vector parallel to the $q_1$-axis. From this fact, Lemma~\ref{gelm-obser}, and Definition~\ref{def-linearized-upper-cap}
it follows that for a centrally symmetric convex body $K \in {\mathcal K}^{2n}$
 \begin{equation} \label{lower-cound-linear-cy1}
  \overline c_{\rm lin}(K)  \leq  4   \inf_{v \in {\mathbb S}^{2n-1}} \ \inf_{w \, : \, \langle Jv,w \rangle=1} \, \|v\|_{K^{\circ}} \|w \| _{K_{v}^{\circ} } =  4  \inf_{v \in {\mathbb S}^{2n-1}} \|v\|_{K^{\circ}} \inf_{w \, : \,  \langle Jv,w \rangle=1} \|w \|_{K_{v}^{\circ} }.
\end{equation}
%
We focus now on the second infimum on the right-hand side of~$(\ref{lower-cound-linear-cy1})$.
Note that for a fixed vector $v \in {\mathbb S}^{2n-1}$, the equality $\langle Jv,w \rangle=1$ is equivalent to $\langle Jv, w-Jv \rangle=0$. Denoting  $z: = w-Jv$, we can write
\begin{equation} \label{changing-var} \inf_{w \, : \,  \langle Jv,w \rangle=1} \|w \| _{K_{v}^{\circ} }  = \inf_{z \, : \, z \perp Jv} \|   Jv+z \|_{K_{v}^{\circ}}. \end{equation}
%
%
This quantity measures the distance, with respect to the semi-metric induced by $\| \cdot \|_{K_v^{\circ}}$, between the vector $Jv$ and the subspace $\{ Jv \}^{\perp}$ orthogonal to it. Using the Hahn--Banach theorem we obtain 
\begin{equation} \label{HB-arg1}
 \inf_{z \, : \, z \perp Jv} \|   Jv+z \|_{K_{v}^{\circ}} = {\rm dist}_{\| \cdot \|_{K_{v}^{\circ}}} (Jv, \{ Jv \}^{\perp}) = \sup_{u} \langle u, Jv \rangle,
\end{equation}
where the supremum is taken over all vectors
$u$ such that $u \in {\rm span} \{Jv \}$ and $\|u \|_{K_{v}} \leq 1$.  Note that we have used the fact that $ {({K_{v}^{\circ}})^{\circ}} =K_{v} $. Next, we use the fact that $Jv$ is orthogonal to $v$ (and hence in particular $ \| Jv \|_{K_v} < \infty$) to deduce from~$(\ref{changing-var})$ and~$(\ref{HB-arg1})$ that
\begin{equation} \label{HB-arg2}
  \inf_{w \, : \,  \langle Jv,w \rangle=1} \|w \| _{K_{v}^{\circ} }  = \sup_{u \in {\rm span}\{Jv\}, \ \| u \|_{K_{v}} \leq 1} \langle u, Jv \rangle \leq {\frac {\langle Jv,Jv \rangle} {\| Jv \|_{K_{v}}}} = {\frac {1} {\| Jv \|_{K}}}.
\end{equation}
From the combination of~$(\ref{lower-cound-linear-cy1})$ and~$(\ref{HB-arg2})$ we obtain that 
 \begin{equation} \label{lower-cound-linear-cy4}
  \overline c_{\rm lin}(K) \leq  4  \inf_{v \neq 0} \, {\frac {\| v \|_{K^{\circ}}} {\|Jv \|_K} } = {\frac {4} {\|J\|_{_{{K^{\circ} \rightarrow K}}}}}, \end{equation}
which completes the proof of the proposition.
 \end{proof}
 \begin{remark} \label{rmk-about-equi-of-lin-cy-cap-of-K-and-K-K}
{\rm For a general convex body $K$ in ${\mathbb R}^{2n}$ (not necessarily centrally symmetric), the same proof as the one above will give the following bound:  
\begin{equation} \label{upp-bound-non-sym}  {\frac 1 4} \,  \overline c_{\rm lin}(K - K) \leq  \overline c_{\rm lin}(K) \leq  {\frac 1 {\|J\|_{_{(K-K)^{\circ} \rightarrow (K-K)}}} }. \end{equation} On the other hand, from Proposition~\ref{lower-bound-for-ECH-cap} it follows that \begin{equation} \label{lower-bound-non-sym} c_{_{\rm EHZ}}(K) \geq  \sup_{v} {\frac 1 {\|J\|_{(K-v)^{\circ} \rightarrow (K-v)}} },  \end{equation}
where the supremum is taken over all $v \in {\mathbb R}^{2n}$ such that $v \in {\rm Int}(K)$. 
We remark that although the  upper bound for $ \overline c_{\rm lin}(K)$ in~$(\ref{upp-bound-non-sym})$, and the lower bound for $c_{_{\rm EHZ}}(K)$ in~$(\ref{lower-bound-non-sym})$ seem not too far away, we do not expect them to be asymptotically equivalent in general. 
 }
 \end{remark}
\begin{proof}[{\bf Proof of Theorem~\ref{main-thm}}] 
For a smooth symmetric convex body $K$, the proof  follows immediately from Propositions~\ref{lower-bound-for-ECH-cap} and~\ref{upper-bound-for-cylindrical-cap}. The general case (i.e., without the smoothness assumption) follows by a standard approximation argument, as indicated at the beginning of this section.  
\end{proof}

\section{Linearized Symplectic Capacities} \label{Sec-linearized}
In this section we prove Theorem~\ref{thm-about-lin-capacities}.
We recall first the following definition. 
\begin{definition} \label{def-linearized-Gromov-width} The linearized Gromov width ${\underline c}_{\rm lin}$ of a set ${\mathcal U} \subset {\mathbb R}^{2n}$ is defined as
$$ {\underline c}_{\rm lin}({\mathcal U} ) = \sup_{S} \{ \pi r^2 \, | \, SB^{2n}(r) \subset {\mathcal U}\},$$
where the supremum is taken over all affine symplectic maps $S \in {\rm ISp}(2n)$. 
\end{definition}
The following is the main ingredient in the proof of  Theorem~\ref{thm-about-lin-capacities}. 
\begin{proposition} \label{cap-rotated-cube}
Let $Q=[-1,1]^{2n}$ be the standard cube in ${\mathbb R}^{2n}$. Then, for every orthogonal transformation $O \in {\rm O}(2n)$ 
one has $\underline c_{\rm lin}( {O}Q) \leq \pi$. Moreover, there is a rotation ${\widetilde O} \in {\rm O}(2n)$ for which
 $\underline c(\widetilde {O}Q) \geq \sqrt{n/2}$.
\end{proposition}
\begin{proof}[{\bf Proof of Proposition~\ref{cap-rotated-cube}}]
Note first that for every  orthogonal transformation $O \in {\rm O}(2n)$, 
\begin{equation}\label{upper-bound-rotated-cube} \underline c_{\rm lin}( {O}Q) \leq \sup_{L} \{ \pi r^2 \, | \, LB^{2n}(r) \subseteq Q\}, \end{equation}
where the supremum is taken over all affine volume-preserving linear maps $L$ of ${\mathbb R}^{2n}$.
It is straightforward to check that the largest ellipsoid contained in the cube $Q$ is the unit-ball $B^{2n}(1)$, and hence 
$\underline c_{\rm lin}( {O}Q) \leq \pi$ for every orthogonal transformation $O \in {\rm O}(2n)$. 

For the second part of the proposition, consider the Lagrangian splitting  ${\mathbb R}^{n}({q}) \times  {\mathbb R}^{n}({p})$ of  ${\mathbb R}^{2n}$, and the following configuration: $B^n_{\infty} (\alpha) \times B^n_1(\beta) \subset {\mathbb R}^{2n}$, where
\begin{eqnarray*} B^n_{\infty}(\alpha) & = & \left \{ (q_1,\ldots,x_q) \in {\mathbb R}^{n}({q}) \, | \, \max \{ |q_1|,\ldots,|q_n| \} < \alpha \right \},  \\ 
 B^n_1(\beta) & = & \{ (p_1,\ldots,p_n)  \in {\mathbb R}^{n}({p}) \, | \, \sum_{i=1}^n |p_i|  < \beta  \}.
\end{eqnarray*}
Note that  $B^n_{\infty}(1) \times B^n_1(1)$  is the product of a hypercube and its dual body, the cross-polytope. It is  known (see e.g., \textsection{4} of~\cite{LMS}) that for every $\varepsilon > 0$, the ball $B^{2n}(r)$ symplectically embeds (via a non-linear symplectomorphism) into the product $B^n_{\infty}(1) \times B^n_1(\beta(1+\varepsilon))$,
for a parameter $\beta$ such that ${\rm Vol}(B^n_{\infty}(1) \times B^n_1(\beta)) = {\rm Vol}(B^{2n}(r))$.
 In particular, this implies that 
for an orthogonal transformation $O$ of ${\mathbb R}^{2n}$, one has the following lower bound
$$ c(OQ) \geq \sup\{4r \, | \, B^n_{\infty}(1) \times B^n_1(r) \subseteq OQ \}.$$
Thus, to complete the proof of the proposition it is enough to find an orthogonal transformation $ O \in {\rm O}(2n)$   such that 
$$ O( B^n_{\infty}(1) \times  B^n_1(\widetilde r) ) \subseteq Q, $$
where $\widetilde r >  \sqrt{n/2}$. 
In particular, it is enough to find an orthogonal transformation $O'$ of ${\mathbb R}^n({p})$ such that $O' ( B^n_1(\widetilde r)) \subseteq [-1,1]^n \subset {\mathbb R}^n(p)$, with $\widetilde r$ as above. The fact that such a transformation exists is well known to experts. For completeness we will give  an explicit construction\footnote[3]{The first named author learned this example from R.S. Ismagilov around 1976.}.
We define the elements $O'_{kj}$ of the matrix $O'$ by 
 \begin{equation} \label{the-matrix}
\sqrt{n} \, O'_{kj} = \left\{
\begin{array}{ll}
 {\sqrt{2}} \sin \left({\frac {kj} n} 2 \pi \right) & \text{for } \, 1 \leq k < {\frac n 2} \ {\rm and} \ 1\leq j \leq n,\\
(-1)^j & \text{for } \, k = {\frac n 2} \ {\rm and} \ 1 \leq j \leq n, \\
{\sqrt{2}}  \cos \left({\frac {kj} n} 2 \pi \right) & \text{for } \,   {\frac n 2} < k < n \ {\rm and} \ 1\leq j \leq n, \\
1 & \text{for } \, k =n \ {\rm and} \ 1\leq j \leq n.
\end{array} \right.
\end{equation}
It is a straightforward computation (based on the orthonormality of the standard Fourier basis) to check that the matrix $O'$ defined by~$(\ref{the-matrix})$ is indeed an orthogonal transformation. Moreover, denote by $\{e_i\}_{i=1}^{n}$ the standard basis
of ${\mathbb R}^n(p)$.  Note that  $B^n_1(1) = {\rm Conv} \{ \pm e_i \}$. It follows immediately from the definition of the matrix $O$' that 
$$ \| O' e_i \|_{\infty} := \max_ {1 \leq j \leq n} | (O'e_i)_j | \leq {\frac {\sqrt{2}} {\sqrt{n}}},$$
where $(O'e_i)_j$ stands for the $j$-th component of the vector $O'e_i \in {\mathbb R}^n(p)$. This implies in particular that
$$ O' ( B^n_1(\sqrt{n})) = O' ({\rm Conv} \{ \pm \sqrt{n} e_i\} ) = {\rm Conv} \{ \pm O' \sqrt{n} e_i \}\subseteq [-\sqrt{2},\sqrt{2}]^n,$$
which completes the proof of Proposition~\ref{cap-rotated-cube}. 
\end{proof}
\begin{proof}[{\bf Proof of Theorem~\ref{thm-about-lin-capacities}}]
Note that an immediate corollary from Propositions~\ref{lower-bound-for-ECH-cap} and~\ref{upper-bound-for-cylindrical-cap} is that the cylindrical capacity $\overline c$ is asymptotically equivalent to its linearized version ${\overline c}_{\rm lin}$ for symmetric convex domains in ${\mathbb R}^{2n}$, i.e., for every symmetric convex body $K \in {\mathcal K}^{2n}$,
$$ \overline c(K) \leq  \overline c_{\rm lin}(K) \leq 4  c_{_{\rm EHZ}}(K) \leq 4  \overline c(K).$$
This establishes the first part of Theorem~\ref{thm-about-lin-capacities}. The second part follows from Proposition~\ref{cap-rotated-cube}. 
\end{proof}

\smallskip \noindent
Efim ben David Gluskin \\
School of Mathematical Sciences \\
Tel Aviv University, Tel Aviv 69978, Israel  \\
{\it e-mail}: gluskin@post.tau.ac.il \\

\smallskip \noindent
Yaron Ostrover \\
School of Mathematical Sciences \\
Tel Aviv University, Tel Aviv 69978, Israel  \\
{\it e-mail}: ostrover@post.tau.ac.il \\


\begin{thebibliography}{}

\bibitem{AAKO} Artstein-Avidan, S., Karasev, R. N., Ostrover, Y.  {\it From symplectic measurements to the
Mahler conjecture}, Duke Math. Jour., 163, 2003--2022, (2014).

\bibitem{AAO} Artstein-Avidan, S., Ostrover, Y. {\it Bounds for Minkowski billiard trajectories in convex
bodies}, Intern. Math. Res. Not. (IMRN) (2012) doi:10.1093/imrn/rns216.

\bibitem{CHLS} Cieliebak, T., Hofer, H., Latschev, J., Schlenk
F. {\it Quantitative symplectic geometry,} In: Dynamics, ergodic
theory, and geometry, 1--44,  Math. Sci. Res. Inst. Publ., 54,
Cambridge Univ. Press, Cambridge 2007.

\bibitem{EH1} Ekeland, I., Hofer, H. {\it Symplectic topology and Hamiltonian
dynamics}, Math. Z. 200, 355--378 (1989).
%
\bibitem{EH2} Ekeland, I., Hofer, H. {\it Symplectic topology and Hamiltonian
dynamics  {\rm II}}, Math. Z. 203, 553--567 (1990).

\bibitem{FGS} Frauenfelder, U., Ginzburg, V., Schlenk, F. {\it Energy capacity
inequalities via an action selector},  In: Geometry, spectral theory,
groups, and dynamics, 129--152, Contemp. Math., 387, Amer. Math.
Soc., Providence, RI, 2005.

\bibitem{FH} Floer, A.,  Hofer, H. {\it Symplectic homology. {\rm I}. Open sets in ${\mathbb C}^n$,} Math. Z. 215, no. 1, 37--88, (1994).

\bibitem{FHW} Floer, A.,  Hofer, H., Wysocki, K.  {\it  Applications of symplectic homology. {\rm I},}  Math. Z. 217, no. 4, 577--606, (1994).



\bibitem{G} Gromov, M. {\it Pseudoholomorphic curves in symplectic manifolds},
Invent. Math. 82, no. 2, 307--347, (1985).


\bibitem{Her} Hermann, D. {\it Non-equivalence of symplectic capacities for open
sets with restricted contact type boundary}. Pr\'epublication
d'Orsay num\'ero 32 (29/4/1998).

\bibitem{H1} Hofer, H. {\it On the topological properties of symplectic
maps,} Proc. Roy. Soc. Edinburgh Sect. A 115, 25--38 (1990).

\bibitem{H2} Hofer, H. {\it Symplectic capacities}, 
In: Geometry of low-dimensional manifolds, 2 (Durham, 1989), 15--34, 
London Math. Soc. Lecture Note Ser., 151, Cambridge Univ. Press, Cambridge, 1990. 

\bibitem{HZ} Hofer, H.,  Zehnder, E. {\it Symplectic Invariants
and Hamiltonian Dynamics}, Birkhauser Advanced Texts, Birkhauser
Verlag, 1994.

\bibitem{HZ1} Hofer, H., Zehnder, E. {\it  A new capacity for symplectic
manifolds}, In; Analysis, et cetera, 405--427, Academic Press, Boston,
MA, 1990.

\bibitem{Hu1} Hutchings, M. {\it Quantitative embedded contact homology}, J. Diff. Geom., 88, 231--266, (2011).

%
\bibitem{LMS} Latschev, J., McDuff, D., Schlenk, F. {\it The Gromov width of 4-dimensional tori}, Geom. Topol. 17, 2813--2853, (2013).
%
%
\bibitem{Mc1} McDuff, D. {\it Symplectic topology today},  2014 AMS Joint Mathematics Meeting.

\bibitem{MS} McDuff, D., Salamon, D. {\it Introduction to Symplectic Topology,} 2nd edition, Oxford
University Press, Oxford, England 1998.

%
%



\bibitem{Oh} Oh, Y-G. {\it Chain level Floer theory and Hofer's geometry of the Hamiltonian diffeomorphism
group,} Asian J. Math. 6, no. 4, 579--624, (2002).

\bibitem{O} Ostrover, Y. {\it When symplectic topology meets Banach space geometry}, 
In: Proceedings
of the International Congress of Mathematicians, Seoul 2014, Vol II, pp.
959--981. Jang, S.Y.; Kim,
Y.R., Lee, D.-W.; Yie, I. (Eds.), Kyung Moon SA Co. Ltd, Seoul, Korea 2014. 

\bibitem{Schle} Schlenk. F. {\it Embedding Problems in Symplectic Geometry,}  de Gruyter Expositions in Mathematics, 40, Berlin, 2005.

\bibitem{Schn} Schneider, R. {\it Convex Bodies: the Brunn-Minkowski Theory,} 
Second expanded edition. Encyclopedia of Mathematics and its Applications, 151. Cambridge University Press, Cambridge, 2014. 

\bibitem{Sib} Siburg, K. F. {\it Symplectic capacities in two dimensions}, Manuscripta Math. 78, no. 2, 149--163, (1993).

\bibitem{RS} Rogers, C. A., Shephard, G. C. {\it Convex bodies associated with a given convex body}, J. London Soc. 33, 270-281, (1958).


\bibitem{V} Viterbo, C. {\it Metric and isoperimetric problems in symplectic geometry.} J.
Amer. Math. Soc. 13, no. 2, 411--431, (2000).


\bibitem{V2} Viterbo, C. {\it Symplectic topology as the geometry of
generating functions,} Math. Ann. 292, 685--710, (1992).

\end{thebibliography}
\end{document}